\documentclass[12pt]{article}
\usepackage{amsmath,amssymb}

\newtheorem{theorem}{Theorem}[section]

\newtheorem{definition}[theorem]{Definition}

\newtheorem{remark}[theorem]{Remark}

\newenvironment{proof}{\smallskip\par{\sc Proof.}\enspace}%
 {{\unskip\nobreak\hfil\penalty50\hskip2em
          \hbox{}\nobreak\hfil{\rule[-1pt]{5pt}{10pt}}
          \parfillskip=0pt\finalhyphendemerits=0
          \par\medskip}} 

\makeatletter
\def\section{\@startsection {section}{1}{\z@}{3.25ex plus 1ex minus
 .2ex}{1.5ex plus .2ex}{\large\bf}}
\def\subsection{\@startsection{subsection}{2}{\z@}{3.25ex plus 1ex minus
 .2ex}{1.5ex plus .2ex}{\normalsize\bf}}
\@addtoreset{equation}{section} 
\makeatother

\begin{document}

\begin{center}
\LARGE\textsf{A unified approach to local limit theorems in Gaussian spaces and the law of small numbers}
\end{center}

\vspace*{.3in}

\begin{center}
\sc
\large{Alberto Lanconelli\footnote{Dipartimento di Matematica, Universit\'a degli Studi di Bari Aldo Moro, Via E. Orabona 4, 70125 Bari - Italia. E-mail: \emph{alberto.lanconelli@uniba.it}}}\\
\end{center}

\vspace*{.3in}

\begin{abstract}
Through a reformulation of the local limit theorem and law of small numbers, which is obtained by working in the spaces naturally associated to the limiting distributions, we discover a general and abstract framework for the investigation of that type of limit theorems. From this new perspective, the convolution and scaling operators utilized in the classical results mentioned before will be identified with the Wick product and second quantization operators, respectively. And here is the advantage of our approach: definitions and most of the properties of Wick products and second quantization operators do not depend (\emph{mutatis mutandis}) on the underlying probability measure. Then, with the help of H\"older-Young-type inequalities for Gaussian  and Poisson Wick products proved in previous papers, we show the $\mathcal{L}^1$ convergence of the densities towards the desired limit. We remark that our approach extends without additional assumptions to infinite dimensional Gaussian spaces.
   
\end{abstract}

\bigskip


\bigskip


\section{Introduction}

The aim of the present paper is to propose a unified framework for proving the two following well known limit theorems:
\begin{theorem}[Local Limit Theorem]\label{LLT}
Let $\{X_n\}_{n\geq 1}$ be a sequence of independent and identically distributed real valued random variables. Assume that the common law of the $X_n$'s is absolutely continuous with respect to the one dimensional Lebesgue measure and that $E[X_n]=0$ and $Var(X_n)=1$ for all $n\geq 1$. \\
Then the density of $\frac{X_1+\cdot\cdot\cdot+X_n}{\sqrt{n}}$ converges in $\mathcal{L}^1(\mathbb{R})$, as $n$ goes to infinity, to the one dimensional standard Gaussian density.
\end{theorem}
\begin{theorem}[Law of Small Numbers]\label{LSN}
Let $\{X_n\}_{n\geq 1}$ be a sequence of independent and identically distributed random variables taking values on $\mathbb{N}_0$. Assume that $E[X_n]=a$ for some $a>0$ and all $n\geq 1$.\\ 
Then the law of $T_{\frac{1}{n}}X_1+\cdot\cdot\cdot+T_{\frac{1}{n}}X_n$ converges, as $n$ goes to infinity, to the Poisson distribution with intensity $a>0$.
(Here, $\mathbb{N}_0:=\mathbb{N}\cup\{0\}$ and for $\alpha\in[0,1]$, $T_\alpha X$ denotes the $\alpha$-thinning of the law of the random variable $X$. See Definition \ref{thinning} below.)
\end{theorem}
Theorem \ref{LLT} was obtained by Prokhorov in \cite{Prokhorov}; Ranga Rao and Varadarajan \cite{RRV} proved point-wise convergence of the densities while Gnedenko \cite{Gnedenko} studied uniform convergence. We also mention the work of Barron \cite{Barron} where the relative entropy (or Kullback-Leibler divergence) of $\frac{X_1+\cdot\cdot\cdot+X_n}{\sqrt{n}}$ with respect to the standard Gaussian measure is shown to converge to zero (monotonically along a certain subsequence).\\
Theorem \ref{LSN} is a generalization of the famous Binomial-to-Poisson convergence. See for instance Barbour et al. \cite{BHJ} for a systematic treatment of Poisson approximation. In Harremo\"es et al. \cite{HJK} the authors proved an information theoretic version of Theorem \ref{LSN} in the spirit of the result of Barron \cite{Barron}.\\
The main novelty of our approach consists in utilizing a common framework for the representation of the density of the random variable $\frac{X_1+\cdot\cdot\cdot+X_n}{\sqrt{n}}$ in Theorem \ref{LLT} and  the distribution of $T_{\frac{1}{n}}X_1+\cdot\cdot\cdot+T_{\frac{1}{n}}X_n$ in Theorem \ref{LSN}. This common framework is based on the one dimensional version of two basic tools from the analysis on Wiener spaces: the Wick product, which will play the role of the convolution product, and the second quantization operator (also known as Ornstein-Uhlenbeck semigroup), which will play the role of the scaling operator (we refer the reader to Janson \cite{J} for all the details on these tools). The crucial ingredient in our construction is to reformulate the problems in their natural frameworks: in the case of the Local Limit Theorem, we work with densities with respect to the standard Gaussian measure while, in the case of the Law of Small Numbers, we investigate the behavior of the Radon-Nikodym derivatives with respect to the Poisson distribution with intensity $a>0$. This idea is similar to the one proposed in Barron \cite{Barron} for the Gaussian measure and in Harremo\"es et al. \cite{HJK} for the Poisson distribution.\\
Once the densities are taken with respect to the right reference measures, then the convergence theorems can be formulated in the same manner (see Theorem \ref{main theorem} and Theorem \ref{main theorem poisson} below) and proved with the help of general inequalities for Wick products and second quantization operators (see Theorem \ref{gaussian young inequality} and Theorem \ref{poisson young inequality} below). We obtain $\mathcal{L}^1$ convergence of the densities towards the constant function one, that corresponds to the Radon-Nikodym derivative (with respect to the reference measure) of the desired limit. Morever, the moment conditions ($E[X_n]=0$ and $Var(X_n)=1$ for the local limit theorem and $E[X_n]=a$ for the law of small Numbers) take a very natural form: they are equivalent to the orthogonality between the Radon-Nikodym density of the $X_n$ and the first and second orthogonal polynomials associated to the reference measure (see the proofs of Theorem \ref{main theorem} and Theorem \ref{main theorem poisson} below). The idea of the proofs of the main theorems is taken from \cite{LS2015}. \\
However, we need to pay a price for the generality of our approach. We are in fact able to prove the convergence for a smoothed version of the quantity investigated in the classical version of the theorems.  Nevertheless, our technique is, at least in the Gaussian case, dimension independent (see the Remark \ref{infinite dimensions} and Remark \ref{assumptions} below).\\
The paper is organized as follows: Section 2 deals with the local limit theorem while Section 3 is devoted to the law of small numbers. The two sections are structured in exactly the same manner with the aim of underlying the unified nature of the approach. For more information on orthogonal polynomials we refer the reader to Chihara \cite{Chihara} and to Szeg\"o\cite{Szego}.

\section{Gaussian local limit theorem}

Let $\mu$ denote the standard one dimensional Gaussian measure, i.e.
\begin{eqnarray*}
\mu(A):=\int_A\frac{1}{\sqrt{2\pi}}\exp\Big\{-\frac{x^2}{2}\Big\}dx,\quad A\in\mathcal{B}(\mathbb{R})
\end{eqnarray*}
 and consider the real Hilbert space $\mathcal{L}^2(\mathbb{R},\mathcal{B}(\mathbb{R}),\mu)$ 
together with the family of monic Hermite polynomials $\{h_n\}_{n\geq 0}$. These polynomials constitute an orthogonal basis for the space $\mathcal{L}^2(\mathbb{R},\mathcal{B}(\mathbb{R}),\mu)$ and therefore any element $f$ of that space can be represented as
\begin{eqnarray*}
f(x)=\sum_{j\geq 0}\gamma_jh_j(x).
\end{eqnarray*}
Note that 
\begin{eqnarray*}
\int_{\mathbb{R}}f^2(x)d\mu(x)=\sum_{j\geq 0}j!|\gamma_j|^2,
\end{eqnarray*}
since the polynomials are not normalized. We now introduce a scaling operator and an unbounded bilinear multiplication on $\mathcal{L}^2(\mathbb{R},\mathcal{B}(\mathbb{R}),\mu)$; they will be our fundamental tools in proving the limit theorems under consideration. For $\lambda\in [-1,1]$, the \emph{second quantization operator} is the bounded linear operator
\begin{eqnarray*}
\Gamma(\lambda):\mathcal{L}^2(\mathbb{R},\mathcal{B}(\mathbb{R}),\mu)&\to& \mathcal{L}^2(\mathbb{R},\mathcal{B}(\mathbb{R}),\mu)\\
\sum_{j\geq 0}\gamma_jh_j(x)&\mapsto& \sum_{j\geq 0}\lambda^j\gamma_jh_j(x).
\end{eqnarray*}
This map can be extended to a bounded linear operator from $\mathcal{L}^p(\mathbb{R},\mathcal{B}(\mathbb{R}),\mu)$ to $\mathcal{L}^p(\mathbb{R},\mathcal{B}(\mathbb{R}),\mu)$ for any $p\in[1,+\infty]$ with norm equal to one. Moreover for $f,g\in\mathcal{L}^2(\mathbb{R},\mathcal{B}(\mathbb{R}),\mu)$ with
\begin{eqnarray*}
f(x)=\sum_{j\geq 0}\gamma_jh_j(x)\quad\mbox{ and }\quad g(x)=\sum_{j\geq 0}\delta_jh_j(x),
\end{eqnarray*}
we define the \emph{Wick product} of $f$ and $g$ to be the function
\begin{eqnarray*}
(f\diamond g)(x):=\sum_{j\geq 0}\Big(\sum_{i=0}^jf_ig_{j-i}\Big)h_j(x).
\end{eqnarray*}
The Wick product, combined with second quantization operators, possesses the following clear probabilistic interpretation. (This fact is implicitly contained in \cite{DLS} and further analyzed in \cite{LSportelli}, where a similar statement for the chi-square distribution is obtained).

\begin{theorem}\label{gaussian wick interpretation}
Let $X_1,...,X_n$ be independent real valued random variables and denote by $\mu_{X_1},...,\mu_{X_n}$ the corresponding laws on $\mathbb{R}$, respectively. Assume that the measures $\mu_{X_1},...,\mu_{X_n}$
are absolutely continuous with respect to $\mu$. Then, for any $\alpha_1,...\alpha_n\in [-1,1]$ such that
\begin{eqnarray*}
\alpha_1^2 + \cdot\cdot\cdot+\alpha_n^2 =1,
\end{eqnarray*}
we have
\begin{eqnarray}\label{gaussian wick}
\Gamma(\alpha_1)\frac{d\mu_{X_1}}{d\mu}\diamond\cdot\cdot\cdot\diamond\Gamma(\alpha_n)\frac{d\mu_{X_n}}{d\mu} & = & \frac{d\mu_{\alpha_1 X_1+\cdot\cdot\cdot+\alpha_n X_n}}{d\mu},
\end{eqnarray}
where $\frac{dQ}{d\mu}$ denotes the Radon--Nikodym derivative of the
measure $Q$ with respect to the reference measure $\mu$ and $\mu_{\alpha_1 X_1+\cdot\cdot\cdot+\alpha_n X_n}$ denotes the law of the random variable $\alpha_1 X_1+\cdot\cdot\cdot+\alpha_n X_n$.
\end{theorem}

\noindent With this probabilistic interpretation in mind, the next theorem can be considered as a Gaussian Young inequality (see \cite{DLS} for the proof).
Here and for the rest of the present section, the symbol $\Vert\cdot\Vert_p$ will denote the usual norm in $\mathcal{L}^{p}(\mathbb{R},\mathcal{B}(\mathbb{R}),\mu)$.
\begin{theorem}\label{gaussian young inequality}
Let $\alpha_1,...,\alpha_n\in [-1,1]$ be such that $\alpha_1^2+\cdot\cdot\cdot+\alpha_n^2 = 1$ and let
$p_1,...,p_n,r \in [1$, $+\infty]$ satisfy the following condition
\begin{eqnarray*}
\frac{\alpha_1^2}{p_1-1}+\cdot\cdot\cdot+\frac{\alpha_n^2}{p_n-1}=\frac{1}{r-1}.
\end{eqnarray*}
If $f_i\in\mathcal{L}^{p_i}(\mathbb{R},\mathcal{B}(\mathbb{R}),\mu)$ for each $i=1,...,n$, then $\Gamma(\alpha_1)f_1 \diamond
\cdot\cdot\cdot\diamond\Gamma(\alpha_n)f_n\in\mathcal{L}^r(\mathbb{R},\mathcal{B}(\mathbb{R}),\mu)$. More precisely,
\begin{eqnarray}\label{gaussian young}
\parallel\Gamma(\alpha_1)f\diamond\cdot\cdot\cdot\diamond\Gamma(\alpha_n)f_n\parallel_r & \leq &
\parallel f_1\parallel_{p_1}\cdot\cdot\cdot\parallel f_n\parallel_{p_n}.
\end{eqnarray}
\end{theorem}

\noindent We are now ready for the main result of the present section. The theorem is stated and proved in the one dimensional Gaussian framework introduced above; however, its validity can be extended without additional assumptions to infinite dimensional abstract Wiener spaces (see \cite{LS2015}).
\begin{theorem}\label{main theorem}
Let $\{X_n\}_{n\geq 1}$ be a sequence of real valued, independent and identically distributed  random variables with zero mean and unit variance. Suppose that the common law of the $X_n$'s is absolutely continuous with respect to the measure $\mu$ with a density belonging to $\mathcal{L}^2(\mathbb{R},\mu)$. Then, for any non negative sequence $\{b_n\}_{n\geq 1}$ such that
\begin{eqnarray}\label{sequence}
\quad \lim_{n\to +\infty}\frac{b_n}{n}=0\quad\mbox{ and }\quad\lim_{n\to +\infty}\frac{b_n}{n^{\frac{2}{3}}}=+\infty,
\end{eqnarray} 
the density (with respect to $\mu$) of  
\begin{eqnarray}\label{limit}
\sqrt{\frac{n}{n+b_n}}\cdot\frac{X_1+\cdot\cdot\cdot+X_n}{\sqrt{n}}+\sqrt{\frac{b_n}{n+b_n}}\cdot Z
\end{eqnarray}
converges in $\mathcal{L}^1(\mathbb{R},\mathcal{B}(\mathbb{R}),\mu)$ to $1$ as $n$ tends to infinity. Here $Z$ is a standard Gaussian random variable which is independent of the sequence $\{X_n\}_{n\geq 1}$. 
\end{theorem}

\begin{remark}\label{infinite dimensions}
The classic local limit theorem is about the convergence in a specified topology of the density of $\frac{X_1+\cdot\cdot\cdot+X_n}{\sqrt{n}}$ towards the standard Gaussian density. We are working in a space with a Gaussian reference measure and hence we are interested in the convergence to the constant function $1$. We prove the convergence of the density of the random variable in (\ref{limit}) which is a smoothed version of $\frac{X_1+\cdot\cdot\cdot+X_n}{\sqrt{n}}$. Nevertheless, the validity of our approach is dimension independent as mentioned above.
\end{remark}

\begin{remark}\label{assumptions}
Theorem \ref{main theorem} is a refinement of the result proved in \cite{LS2015}. In fact, in that paper it is assumed that $b_n=\alpha n$ for some $\alpha>0$; with this choice (that does not satisfy the first condition in (\ref{sequence})) the extra term containing $Z$ in (\ref{limit}) is not vanishing as $n$ tends to infinity. Since the presence of $Z$ in (\ref{limit}) serves to mollify the resulting density, we believe that Theorem \ref{main theorem} improves the result obtained in the paper \cite{LS2015}. 
\end{remark}

\begin{proof}
Let $f$ be the common density of the $X_n$'s with respect to the measure $\mu$. From Theorem \ref{gaussian wick interpretation} we know that the density of $\frac{X_1+\cdot\cdot\cdot+X_n}{\sqrt{n}}$ is given by
\begin{eqnarray*}
\Gamma\Big(\frac{1}{\sqrt{n}}\Big)f\diamond\cdot\cdot\cdot\diamond\Gamma\Big(\frac{1}{\sqrt{n}}\Big)f=\Big(\Gamma\Big(\frac{1}{\sqrt{n}}\Big)f\Big)^{\diamond n}
\end{eqnarray*}
where $g^{\diamond n}$ means $g\diamond\cdot\cdot\cdot\diamond g$ ($n$-times); moreover, the density of the random variable in (\ref{limit}) can be written as
\begin{eqnarray*}
\Gamma\Big(\sqrt{\frac{n}{n+b_n}}\Big)\Big(\Gamma\Big(\frac{1}{\sqrt{n}}\Big)f\Big)^{\diamond n}\diamond\Gamma\Big(\sqrt{\frac{b_n}{n+b_n}}\Big)1=\Big(\Gamma\Big(\frac{1}{\sqrt{n+b_n}}\Big)f\Big)^{\diamond n}.
\end{eqnarray*}
Here we utilized the functorial property $\Gamma(\lambda)(f\diamond g)=(\Gamma(\lambda)g) \diamond(\Gamma(\lambda)g)$ and the identity $\Gamma(\lambda)1=1$ (note that the density of $Z$ with respect to $\mu$ is one). Observe in addition that we can write without ambiguity the right hand side of the previous equation as 
$\Gamma\Big(\frac{1}{\sqrt{n+b_n}}\Big)f^{\diamond n}$ (again as a consequence of the interplay between second quantization operators and Wick product).\\ 
Observe in addition that assuming the $X_n$'s to be with mean zero and unit variance is equivalent to require that $f$ is of the form $f=1+\sum_{k\geq 3}\gamma_kh_k$; in fact
\begin{eqnarray}\label{moments}
0&=&E[X_n]=\int_{\mathbb{R}}xf(x)d\mu(x)=\int_{\mathbb{R}}h_1(x)f(x)d\mu(x)=\gamma_1\\
1&=&Var(X_n)=\int_{\mathbb{R}}x^2f(x)d\mu(x)=\int_{\mathbb{R}}h_2(x)f(x)d\mu(x)+1=2\gamma_2+1
\end{eqnarray}
(recall that $h_1(x)=x$ and $h_2(x)=x^2-1$). Our aim is to prove that
\begin{eqnarray*}
\lim_{n\to +\infty}\Big\Vert \Gamma\Big(\frac{1}{\sqrt{n+b_n}}\Big)f^{\diamond n}-1\Big\Vert_1=0.
\end{eqnarray*}
First of all, exploiting the associativity and distributivity of the Wick product we write 
\begin{eqnarray*}
\Gamma\Big(\frac{1}{\sqrt{n+b_n}}\Big)f^{\diamond
n}-1&=&\sum_{j=1}^n\Gamma\Big(\frac{1}{\sqrt{n+b_n}}\Big)f^{\diamond
j}-\Gamma\Big(\frac{1}{\sqrt{n+b_n}}\Big)f^{\diamond j-1}\\
&=&\sum_{j=1}^n\Gamma\Big(\frac{1}{\sqrt{n+b_n}}\Big)f^{\diamond
j-1}\diamond\Big(\Gamma\Big(\frac{1}{\sqrt{n+b_n}}\Big)f-1\Big).
\end{eqnarray*}
Now take the $\mathcal{L}^1(\mathbb{R},\mathcal{B}(\mathbb{R}),\mu)$-norm and use the triangle
inequality:
\begin{eqnarray*}
\Big\Vert\Gamma\Big(\frac{1}{\sqrt{n+b_n}}\Big)f^{\diamond
n}-1\Big\Vert_1&=&\Big\Vert\sum_{j=1}^n\Gamma\Big(\frac{1}{\sqrt{n+b_n}}\Big)f^{\diamond
j-1}\diamond\Big(\Gamma\Big(\frac{1}{\sqrt{n+b_n}}\Big)f-1\Big)\Big\Vert_1\\
&\leq&\sum_{j=1}^n\Big\Vert\Gamma\Big(\frac{1}{\sqrt{n+b_n}}\Big)f^{\diamond
j-1}\diamond\Big(\Gamma\Big(\frac{1}{\sqrt{n+b_n}}\Big)f-1\Big)\Big\Vert_1.
\end{eqnarray*}
Now apply Theorem \ref{gaussian young inequality} (actually we need only the $\mathcal{L}^1$-form of the inequality which was proven
before in the paper \cite{LS-}) with
\begin{eqnarray*}
\alpha_1=\sqrt{\frac{n-1}{n+b_n}},\quad\mbox{ and }\quad\alpha_2=\sqrt{\frac{b_n+1}{n+b_n}}
\end{eqnarray*}
to get
\begin{eqnarray*}
\Big\Vert\Gamma\Big(\frac{1}{\sqrt{n+b_n}}\Big)f^{\diamond
n}-1\Big\Vert_1&\leq&\sum_{j=1}^n\Big\Vert\Gamma\Big(\frac{1}{\sqrt{n+b_n}}\Big)f^{\diamond
j-1}\diamond\Big(\Gamma\Big(\frac{1}{\sqrt{n+b_n}}\Big)f-1\Big)\Big\Vert_1\\
&\leq&\sum_{j=1}^n\Big\Vert\Gamma\Big(\frac{1}{\sqrt{n-1}}\Big)f^{\diamond
j-1}\Big\Vert_1\cdot\Big\Vert\Gamma\Big(\frac{1}{\sqrt{b_n+1}}\Big)f-1\Big\Vert_1\\
&=&\Big\Vert\Gamma\Big(\frac{1}{\sqrt{b_n+1}}\Big)f-1\Big\Vert_1\cdot\sum_{j=1}^n\Big\Vert\Gamma\Big(\frac{1}{\sqrt{n-1}}\Big)f^{\diamond
j-1}\Big\Vert_1.
\end{eqnarray*}
Observe that employing once again inequality (\ref{gaussian young}) we can bound the last sum as
\begin{eqnarray*}
\sum_{j=1}^n\Big\Vert\Gamma\Big(\frac{1}{\sqrt{n-1}}\Big)f^{\diamond
j-1}\Big\Vert_1&\leq&
\sum_{j=1}^n\Big\Vert\Gamma\Big(\frac{\sqrt{j-1}}{\sqrt{n-1}}\Big)f\Big\Vert_1^{
j-1}\\
&\leq&\sum_{j=1}^n\Vert f\Vert_1^{
j-1}\\
&=& n.
\end{eqnarray*}
Here we are using the fact that $f$ is a density function (in particular is non negative and with integral with respect to $\mu$ equal to one). Therefore
\begin{eqnarray*}
\Big\Vert\Gamma\Big(\frac{1}{\sqrt{n+b_n}}\Big)f^{\diamond
n}-1\Big\Vert_1&\leq&\Big\Vert\Gamma\Big(\frac{1}{\sqrt{b_n+1}}\Big)f-1\Big\Vert_1\cdot\sum_{j=1}^n\Big\Vert\Gamma\Big(\frac{1}{\sqrt{n}}\Big)f^{\diamond
j-1}\Big\Vert_1\nonumber\\
&\leq& n\cdot\Big\Vert\Gamma\Big(\frac{1}{\sqrt{b_n+1}}\Big)f-1\Big\Vert_1\\
&\leq& n\cdot\Big\Vert\Gamma\Big(\frac{1}{\sqrt{b_n+1}}\Big)f-1\Big\Vert_2\\
&=&n\Big(\sum_{k\geq 3}k!\Big(\frac{1}{b_n+1}\Big)^k|\gamma_k|^2\Big)^{\frac{1}{2}},
\end{eqnarray*}
where the $\gamma_k$'s are the components in the 
decomposition of $f$ (recall (\ref{moments})).  Hence,
\begin{eqnarray*}
\Big\Vert\Gamma\Big(\frac{1}{\sqrt{n+b_n}}\Big)f^{\diamond
n}-1\Big\Vert_1&\leq&
n\cdot\Big\Vert\Gamma\Big(\frac{1}{\sqrt{b_n+1}}\Big)f-1\Big\Vert_1 \\
&\leq&n\cdot\Big(\sum_{k\geq 3}k!\Big(\frac{1}{b_n+1}\Big)^k|\gamma_k|^2\Big)^{\frac{1}{2}} \\ 
&\leq&n\Big(\frac{1}{b_n+1}\Big)^{\frac{3}{2}}\Big(\sum_{k\geq
3}k!|\gamma_k|^2\Big)^{\frac{1}{2}}.
\end{eqnarray*}
The last series, being equal to $\Vert f\Vert_2^2-1$, is convergent; we can therefore pass to the limit as $n$ tends to infinity and obtain the desired result. 
\end{proof}

\begin{remark}
The assumption on $f$ to be in $\mathcal{L}^2(\mathbb{R},\mathcal{B}(\mathbb{R}),\mu)$ can be relaxed, using the Nelson's hyper-contractive estimate, to $f\in\mathcal{L}^p(\mathbb{R},\mathcal{B}(\mathbb{R}),\mu)$ for some $p>1$. 
\end{remark}

\section{Law of small numbers}

Let $\nu$ denote the Poisson distribution with intensity $a>0$, i.e.
\begin{eqnarray*}
\nu(\{k\}):=\frac{a^k}{k!}e^{-a},\quad k\in\mathbb{N}_0
\end{eqnarray*}
and consider the real Hilbert space $\mathcal{L}^2(\mathbb{N}_0, 2^{\mathbb{N}_0},\nu)$ together with the family of monic Charlier polynomials $\{c^a_n\}_{n\geq 0}$. These polynomials constitute an orthogonal basis for the space $\mathcal{L}^2(\mathbb{N}_0, 2^{\mathbb{N}_0},\nu)$ and  therefore any element $f$ of that space can be represented as
\begin{eqnarray*}
f(x)=\sum_{j\geq 0}\gamma_jc^a_j(x).
\end{eqnarray*}
Note that 
\begin{eqnarray*}
\sum_{k\geq 0}f^2(k)\nu(\{k\})=\sum_{j\geq 0}a^jj!|\gamma_j|^2,
\end{eqnarray*}
since the polynomials are not normalized. As before, we introduce for $\lambda\in [-1,1]$ the \emph{second quantization operator} as
\begin{eqnarray*}
\Gamma(\lambda):\mathcal{L}^2(\mathbb{N}_0, 2^{\mathbb{N}_0},\nu)&\to& \mathcal{L}^2(\mathbb{N}_0, 2^{\mathbb{N}_0},\nu)\\
\sum_{j\geq 0}\gamma_jc^a_j(x)&\mapsto& \sum_{j\geq 0}\lambda^j\gamma_jc^a_j(x).
\end{eqnarray*}
(We use the same symbol as for its Gaussian analogue since we believe that there will be no danger of confusion). This operator can be extended to a bounded linear operator from $\mathcal{L}^p(\mathbb{N}_0, 2^{\mathbb{N}_0},\nu)$ to $\mathcal{L}^p(\mathbb{N}_0, 2^{\mathbb{N}_0},\nu)$ for any $p\in[1,+\infty]$ with norm equal to one. Moreover, for $f,g\in\mathcal{L}^2(\mathbb{N}_0, 2^{\mathbb{N}_0},\nu)$ with
\begin{eqnarray*}
f(x)=\sum_{j\geq 0}\gamma_jc^a_j(x)\quad\mbox{ and }\quad g(x)=\sum_{j\geq 0}\delta_jc^a_j(x),
\end{eqnarray*}
we define the \emph{Wick product} of $f$ and $g$ to be
\begin{eqnarray*}
(f\diamond g)(x)=\sum_{j\geq 0}\Big(\sum_{i=0}^j\gamma_i\delta_{j-i}\Big)c^a_j(x).
\end{eqnarray*}

\noindent The next definition provides to some extent a discrete analogue of the scaling operation. It was introduced by
R\'enyi in \cite{renyi}. 
\begin{definition}\label{thinning}
Let $X$ be a random variable taking values on $\mathbb{N}_0$ and  denote by $\nu_X$ its law. For any
$\alpha\in [0,1]$ we define the \emph{$\alpha$-thinning} of $\nu_X$ to be the new probability measure on $\mathbb{N}_0$ given by
\begin{eqnarray*}
(T_{\alpha}\nu_{X})(\{k\}) & := & \sum_{n\geq k}{n \choose
k}\alpha^k(1-\alpha)^{n-k}\nu_X(\{n\}),\quad k\in\mathbb{N}_0.
\end{eqnarray*}
With a slight abuse of notation we will write $T_{\alpha}X$ for the $\alpha$-thinning of the law of $X$. Observe that the law of the random variable $T_{\alpha}X$  coincides with the one of
$Y_1 + \cdots + Y_X$, where the
$Y_i$'s are Bernoulli random variables with probability of success
$\alpha$ which are independent of each other and of $X$.  
\end{definition} 
The next two theorems show an interesting connection between $\alpha$-thinning and stochastic independence on one side and Wick product and second quantization operators on the other. They were obtained in \cite{LS} (see also \cite{LSportelli2011}) where the reader is referred for their proofs. Theorem \ref{poisson wick interpretation} and Theorem \ref{poisson young inequality} below represent the Poissonian counterpart of Theorem \ref{gaussian wick interpretation} and Theorem \ref{gaussian young inequality}, respectively. Here and in the sequel, the symbol $\Vert\cdot\Vert_p$ will denote the usual norm in $\mathcal{L}^{p}(\mathbb{N}_0, 2^{\mathbb{N}_0},\nu)$.

\begin{theorem}\label{poisson wick interpretation}
Let $X_1,...,X_n$ be independent random variables taking values on $\mathbb{N}_0$ and denote by $\nu_{X_1},...,\nu_{X_n}$ the corresponding laws, respectively.
Then, for any $\alpha_1,...,\alpha_n\in [0,1]$ such that
\begin{eqnarray*}
\alpha_1+\cdot\cdot\cdot+\alpha_n=1,
\end{eqnarray*}
we have
\begin{eqnarray}\label{poisson wick}
\Gamma(\alpha_1)\frac{d\nu_{X_1}}{d\nu}\diamond\cdot\cdot\cdot\diamond\Gamma(\alpha_n)\frac{d\nu_{X_n}}{d\nu}&=&\frac{d\nu_{T_{\alpha_1}X_1+\cdot\cdot\cdot+T_{\alpha_n}X_n}}{d\nu},
\end{eqnarray}
where $\frac{dQ}{d\nu}$ denotes the Radon--Nikodym derivative of the
measure $Q$ with respect to the reference measure $\nu$ and $\nu_{T_{\alpha_1}X_1+\cdot\cdot\cdot+T_{\alpha_n}X_n}$ stands for the law of the random variable $T_{\alpha_1}X_1+\cdot\cdot\cdot+T_{\alpha_n}X_n$.
\end{theorem}

\begin{theorem}\label{poisson young inequality}
Let $\alpha_1,...,\alpha_n\in [0,1]$ be such that $\alpha_1+\cdot\cdot\cdot+\alpha_n = 1$ and let
$p \in [1$, $+\infty]$. If $f_1,...,f_n\in\mathcal{L}^p(\mathbb{N}_0, 2^{\mathbb{N}_0},\nu)$, then $\Gamma(\alpha_1)f_1\diamond
\cdot\cdot\cdot\diamond\Gamma(\alpha_n)f_n\in \mathcal{L}^p(\mathbb{N}_0, 2^{\mathbb{N}_0},\nu)$. More precisely,
\begin{eqnarray}\label{poisson young}
\parallel\Gamma(\alpha_1)f_1\diamond\cdot\cdot\cdot\diamond\Gamma(\alpha_n)g_n\parallel_p & \leq &
\parallel f_1\parallel_p\cdot\cdot\cdot\parallel f_n\parallel_p.
\end{eqnarray}
\end{theorem}

\noindent We are now able to state and prove the law of small numbers. As the reader will notice, the proof differs only for few details from the proof of Theorem \ref{main theorem}. This confirms the abstract nature of our approach. 

\begin{theorem}\label{main theorem poisson}
Let $\{X_n\}_{n\geq 1}$ be a sequence of independent and identically distributed  random variables taking values on $\mathbb{N}_0$. Suppose that the $X_n$'s have mean $a$ (the intensity of the Poisson distribution) and that $\frac{d\nu_{X_1}}{d\nu}\in\mathcal{L}^2(\mathbb{N}_0,2^{\mathbb{N}_0},\nu)$. Then, for any non negative sequence $\{b_n\}_{n\geq 1}$ such that
\begin{eqnarray}\label{sequence poisson}
\quad \lim_{n\to +\infty}\frac{b_n}{n}=0\quad\mbox{ and }\quad\lim_{n\to +\infty}\frac{b_n}{n^{\frac{1}{2}}}=+\infty,
\end{eqnarray} 
the density (with respect to $\nu$) of  
\begin{eqnarray}\label{limit poisson}
\frac{n}{n+b_n}\cdot (T_{\frac{1}{n}}X_1+\cdot\cdot\cdot+T_{\frac{1}{n}}X_n)+\frac{b_n}{n+b_n}\cdot U
\end{eqnarray}
converges in $\mathcal{L}^1(\mathbb{N}_0,2^{\mathbb{N}_0},\nu)$ to $1$ as $n$ tends to infinity. Here $U$ is a Poissonian random variable with intensity $a$ which is independent of the sequence $\{X_n\}_{n\geq 1}$. 
\end{theorem}

\begin{proof}
Let $f:=\frac{d\nu_{X_1}}{d\nu}$ be the common density of the $X_n$'s with respect to the measure $\nu$. From Theorem \ref{poisson wick interpretation} we know that the density of $T_{\frac{1}{n}}X_1+\cdot\cdot\cdot+T_{\frac{1}{n}}X_n$ is given by
\begin{eqnarray*}
\Gamma\Big(\frac{1}{n}\Big)f\diamond\cdot\cdot\cdot\diamond\Gamma\Big(\frac{1}{n}\Big)f=\Big(\Gamma\Big(\frac{1}{n}\Big)f\Big)^{\diamond n}
\end{eqnarray*}
where $g^{\diamond n}$ means $g\diamond\cdot\cdot\cdot\diamond g$ ($n$-times); moreover, the density of the random variable in (\ref{limit poisson}) can be written as
\begin{eqnarray*}
\Gamma\Big(\frac{n}{n+b_n}\Big)\Big(\Gamma\Big(\frac{1}{n}\Big)f\Big)^{\diamond n}\diamond\Gamma\Big(\frac{b_n}{n+b_n}\Big)1=\Big(\Gamma\Big(\frac{1}{n+b_n}\Big)f\Big)^{\diamond n}.
\end{eqnarray*}
Here we utilized the functorial property $\Gamma(\lambda)(f\diamond g)=(\Gamma(\lambda)g) \diamond(\Gamma(\lambda)g)$ and the identity $\Gamma(\lambda)1=1$ (note that the density of $U$ with respect to $\nu$ is one). Observe in addition that we can write without ambiguity the right hand side of the previous equation as 
$\Gamma\Big(\frac{1}{n+b_n}\Big)f^{\diamond n}$ (again as a consequence of the interplay between second quantization operators and Wick product).\\
As in the Gaussian case, the assumption $E[X_n]=a$ is reflected on $f$ to be of the form $f(x)=1+\sum_{j\geq 2}\gamma_jc_j^a(x)$; in fact, since $c_1^a(x)=x-a$ we have
\begin{eqnarray}\label{poisson moments}
a&=&E[X_n]=\sum_{k\geq 0}k\nu_{X_n}(\{k\})=\sum_{k\geq 0}kf(k)\nu(\{k\})\nonumber\\
&=&\sum_{k\geq 0}c_1^a(k)f(k)\nu(\{k\})+a=a\gamma_1+a.
\end{eqnarray}  
Our aim is to prove that
\begin{eqnarray*}
\lim_{n\to +\infty}\Big\Vert \Gamma\Big(\frac{1}{n+b_n}\Big)f^{\diamond n}-1\Big\Vert_1=0.
\end{eqnarray*}
First of all, exploiting the associativity and distributivity of the Wick product we write 
\begin{eqnarray*}
\Gamma\Big(\frac{1}{n+b_n}\Big)f^{\diamond
n}-1&=&\sum_{j=1}^n\Gamma\Big(\frac{1}{n+b_n}\Big)f^{\diamond
j}-\Gamma\Big(\frac{1}{n+b_n}\Big)f^{\diamond j-1}\\
&=&\sum_{j=1}^n\Gamma\Big(\frac{1}{n+b_n}\Big)f^{\diamond
j-1}\diamond\Big(\Gamma\Big(\frac{1}{n+b_n}\Big)f-1\Big).
\end{eqnarray*}
Now take the $\mathcal{L}^1(\mathbb{N}_0,2^{\mathbb{N}_0},\nu)$-norm and use the triangle
inequality:
\begin{eqnarray*}
\Big\Vert\Gamma\Big(\frac{1}{n+b_n}\Big)f^{\diamond
n}-1\Big\Vert_1&=&\Big\Vert\sum_{j=1}^n\Gamma\Big(\frac{1}{n+b_n}\Big)f^{\diamond
j-1}\diamond\Big(\Gamma\Big(\frac{1}{n+b_n}\Big)f-1\Big)\Big\Vert_1\\
&\leq&\sum_{j=1}^n\Big\Vert\Gamma\Big(\frac{1}{n+b_n}\Big)f^{\diamond
j-1}\diamond\Big(\Gamma\Big(\frac{1}{n+b_n}\Big)f-1\Big)\Big\Vert_1.
\end{eqnarray*}
An application of Theorem \ref{poisson young inequality} with
\begin{eqnarray*}
\alpha_1=\frac{n-1}{n+b_n},\quad\mbox{ and }\quad\alpha_2=\frac{b_n+1}{n+b_n}
\end{eqnarray*}
gives
\begin{eqnarray*}
\Big\Vert\Gamma\Big(\frac{1}{n+b_n}\Big)f^{\diamond
n}-1\Big\Vert_1&\leq&\sum_{j=1}^n\Big\Vert\Gamma\Big(\frac{1}{n+b_n}\Big)f^{\diamond
j-1}\diamond\Big(\Gamma\Big(\frac{1}{n+b_n}\Big)f-1\Big)\Big\Vert_1\\
&\leq&\sum_{j=1}^n\Big\Vert\Gamma\Big(\frac{1}{n-1}\Big)f^{\diamond
j-1}\Big\Vert_1\cdot\Big\Vert\Gamma\Big(\frac{1}{b_n+1}\Big)f-1\Big\Vert_1\\
&=&\Big\Vert\Gamma\Big(\frac{1}{b_n+1}\Big)f-1\Big\Vert_1\cdot\sum_{j=1}^n\Big\Vert\Gamma\Big(\frac{1}{n-1}\Big)f^{\diamond
j-1}\Big\Vert_1.
\end{eqnarray*}
Observe that employing once again inequality (\ref{poisson young}) we can bound the last sum as
\begin{eqnarray*}
\sum_{j=1}^n\Big\Vert\Gamma\Big(\frac{1}{n-1}\Big)f^{\diamond
j-1}\Big\Vert_1&\leq&
\sum_{j=1}^n\Big\Vert\Gamma\Big(\frac{j-1}{n-1}\Big)f\Big\Vert_1^{
j-1}\\
&\leq&\sum_{j=1}^n\Vert f\Vert_1^{
j-1}\\
&=& n.
\end{eqnarray*}
Here we are using the fact that $f$ is a density function (in particular is non negative and with integral with respect to $\nu$ equal to one). Therefore (recall (\ref{poisson moments})),
\begin{eqnarray*}\label{last poisson}
\Big\Vert\Gamma\Big(\frac{1}{n+b_n}\Big)f^{\diamond
n}-1\Big\Vert_1&\leq&\Big\Vert\Gamma\Big(\frac{1}{b_n+1}\Big)f-1\Big\Vert_1\cdot\sum_{j=1}^n\Big\Vert\Gamma\Big(\frac{1}{n}\Big)f^{\diamond
j-1}\Big\Vert_1\nonumber\\
&\leq& n\cdot\Big\Vert\Gamma\Big(\frac{1}{b_n+1}\Big)f-1\Big\Vert_1\\
&\leq& n\cdot\Big\Vert\Gamma\Big(\frac{1}{b_n+1}\Big)f-1\Big\Vert_2\\
&=&n\cdot\Big(\sum_{j\geq 2}a^jj!\Big(\frac{1}{b_n+1}\Big)^{2j}|\gamma_j|^2\Big)^{\frac{1}{2}} \\ 
&\leq&n\cdot\Big(\frac{1}{b_n+1}\Big)^2\Big(\sum_{j\geq
2}a^jj!|\gamma_j|^2\Big)^{\frac{1}{2}}.
\end{eqnarray*}
The last series, being equal to $\Vert f\Vert_2^2-1$, is convergent; we can therefore pass to the limit as $n$ tends to infinity and obtain the desired result. 
\end{proof}

\end{document}